\documentclass[10pt]{amsart}

\usepackage[utf8]{inputenc}
\usepackage[T1]{fontenc}
\usepackage[backend=biber,
isbn=false,
doi=false,
maxbibnames=5,
giveninits=true,
style=alphabetic,
citestyle=alphabetic]{biblatex}
\usepackage{url}
\renewbibmacro{in:}{}
\numberwithin{equation}{section}

\usepackage{eucal}
\usepackage{graphicx} 
\usepackage[inline]{asymptote}
\usepackage{xcolor}
\definecolor{e-mail}{rgb}{0,.40,.80}
\definecolor{reference}{rgb}{0,.40,.80}
\definecolor{citation}{rgb}{0,.40,.80}

\usepackage[colorlinks=true,
linkcolor=reference,
citecolor=citation,
urlcolor=e-mail]{hyperref}
\usepackage[all, 2cell]{xy}
\UseTwocells
\usepackage{mathtools}
\usepackage{tikz-cd}
\usepackage{amssymb}
\usepackage{eucal}
\usepackage[capitalize]{cleveref}
\usepackage{fullpage}

\crefname{subsection}{Subsection}{subsections}

\newtheorem{thm}{Theorem}[section]
\crefname{thm}{Theorem}{theorems}
\newtheorem*{theorem}{Theorem}
\newtheorem{prop}[thm]{Proposition}
\crefname{prop}{Proposition}{propositions}

\newtheorem{cor}[thm]{Corollary}
\newtheorem{lm}[thm]{Lemma}
\theoremstyle{definition}
\newtheorem{defn}[thm]{Definition}
\theoremstyle{remark}
\newtheorem{remark}[thm]{Remark}
\newtheorem{example}[thm]{Example}

\newcommand{\cC}{\mathcal{C}}

\newcommand{\cM}{\mathcal{M}}

\newcommand{\Tr}{\mathrm{Tr}}

\newcommand{\C}{\mathbf{C}}

\newcommand{\U}{\mathrm{U}_{\hbar}}
\newcommand{\Z}{\mathbf{Z}}

\newcommand{\GL}{\mathrm{GL}}

\renewcommand{\b}{\mathfrak{b}}
\newcommand{\g}{\mathfrak{g}}

\newcommand{\m}{\mathfrak{m}}
\newcommand{\p}{\mathfrak{p}}
\newcommand{\w}{\mathfrak{w}}

\newcommand{\Wh}{\mathrm{Wh}_{\hbar}}
\newcommand{\act}{\mathrm{act}}

\DeclareMathOperator{\ad}{ad}

\newcommand{\BiMod}[2]{{}_{#1}\mathrm{BiMod}_{#2}}

\newcommand{\forget}{\mathrm{forget}}

\newcommand{\free}{\mathrm{free}}

\newcommand{\HC}{\mathrm{HC}_{\hbar}}

\newcommand{\id}{\mathrm{id}}

\newcommand{\RMod}{\mathrm{RMod}}

\DeclareMathOperator{\Rep}{Rep}
\newcommand{\res}{\mathrm{res}}

\newcommand{\gl}{\mathfrak{gl}}

\newcommand{\Span}{\mathrm{span}}
\newcommand{\col}{\mathrm{col}}
\newcommand{\row}{\mathrm{row}}
\newcommand{\triv}{\mathrm{triv}}

\newcommand{\cW}{\mathcal{W}}
\renewcommand{\l}{\mathfrak{l}}

\newcommand{\defterm}[1]{\textbf{\emph{#1}}}
\newcommand{\adj}[2]{
	\xymatrix{
		#1 \ar@<.5ex>[r] & #2 \ar@<.5ex>[l]
	}
}

\title{Intertwining operators between subregular Whittaker modules for $\gl_N$ and non-standard quantizations}
\address{Department of Mathematics, Massachusetts Institute of Technology, \newline
	77 Massachusetts Avenue, Cambridge, MA 02139, United States of America}
\email{artkalm@mit.edu}
\author{Artem Kalmykov} 

\address{Mission San Jose High School, \newline
	41717 Palm Avenue, Fremont, CA 94539, United States of America}
\email{twinbrian236@gmail.com}
\author{Brian Li}

\date{\today}

\begin{document}
	\begin{abstract}
		In this paper, we study intertwining operators between subregular Whittaker modules of $\gl_N$ generalizing, on the one hand, the classical exchange construction of dynamical quantum groups, on the other hand, earlier results for principal W-algebras. We explicitly construct them using the generators of W-algebras introduced by Brundan-Kleshchev. We interpret the fusion on intertwining operators in terms of categorical actions and compute the semi-classical limit of the corresponding monoidal isomorphisms which turn out to depend on dynamical-like parameters. 
	\end{abstract}
	
	\maketitle
	
	\section{Introduction}
	
	\renewcommand{\theequation}{\arabic{equation}}
	
	\renewcommand{\U}{\mathrm{U}}
	\renewcommand{\HC}{\mathrm{HC}}
	
	Quantum groups are algebraic objects that arose from the study of the quantum inverse scattering method for solving quantum integrable systems. By now they have become a classical subject with connections to many areas of mathematics: representations in positive characteristic, 3-manifolds and link invariants, 1-dimensional quantum integrable systems.  
	
	In general, it is hard to construct non-trivial quantum groups. In this paper, we adopt a categorical approach via the \emph{Tannakian reconstruction }\cite{Saavedra,UlbrichFibre,UlbrichHopf}: essentially, a quantum group is defined by the collection of its representations and by how one can ``multiply'' them; mathematically, it is formalized by the notion of a \emph{tensor category} with a tensor structure on the forgetful functor to vector spaces which is equivalent to a collection of isomorphisms
	\begin{equation}
		\label{intro::tensor_structure}
		J_{UV} \colon U \otimes V \rightarrow U \otimes V,
	\end{equation}
	satisfying the \emph{twist equation} (see \cite{EGNO}), for all representations $U,V$. The ``quantum'' part of quantum groups usually means that these isomorphisms come in families depending on a quantization parameter, and its first order defines a Poisson-Lie structure, see \cite{EtingofSchiffmannQuantumGroups}. 
	
	One example of such an approach is the \emph{exchange construction} of Etingof-Varchenko \cite{EtingofVarchenkoExchange} that gives rise to \emph{dynamical quantum groups} which is a certain variant of a quantum group depending on additional parameters. Roughly speaking, the authors of \emph{loc. cit.} consider a finite-dimensional analog of \emph{vertex operators} parameterized by a representation of a reductive Lie algebra: on the one hand, such operators have a natural multiplicative structure given by composition, on the other hand, following the state-field correspondence philosophy, they are in bijection with the elements of the representation itself. In particular, the former induces a non-trivial tensor structure on the collection of all the representations. 
	
	Finite-dimensional vertex operators are essentially maps between \emph{Verma modules}; the latter are defined by a highest-weight vector. It was observed in \cite{KalmykovYangians} that in the case of the general linear group $\GL_N$, there is a version of the state-field correspondence for the \emph{Whittaker} modules generated by Whittaker vectors. In particular, we can apply a similar exchange construction; it turns out that it produces a certain \emph{non-standard} quantum group. In fact, the corresponding solution to the \emph{quantum Yang-Baxter equation}, which is closely related to quantum groups, was obtained by Cremmer-Gervais \cite{CremmerGervais} via studying the exchange algebra of the Toda field theory that can be formulated using \emph{affine W-algebras}; at the same time, the natural setup for Whittaker modules is a \emph{(finite) principal W-algebra}.
	
	In general, a W-algebra is associated to a pair of a reductive Lie algebra and a nilpotent element in it, see \cite{Losev} for introduction to the subject. So, one may ask: is there a similar result for other W-algebras? In this paper, we study the exchange construction for the so-called \emph{subregular} nilpotent element $e$ and the corresponding \emph{subregular} W-algebra $\cW$; we recall the definition in \cref{sect::subregular}. We establish a state-field correspondence in this case in \cref{subsect::whittaker_vectors_standard_representation} using remarkable generators of W-algebras for $\gl_N$ introduced by Brundan-Kleshchev \cite{BrundanKleshchev}; we recall the construction of \emph{loc. cit.} in \cref{section:pyramid}. 
	
	Unfortunately, in the subregular case, the exchange construction does \emph{not} give a quantum group. However, there is still a non-trivial tensor structure analogous to \eqref{intro::tensor_structure}. It turns out that the appropriate categorical setup for it is that of \emph{module categories}, see \cref{sect::background}. More precisely, as in \cite{KalmykovSafronov} and \cite{KalmykovYangians}, we formulate the exchange construction in terms of the category of \emph{Harish-Chandra bimodules} and the \emph{finite Drinfeld-Sokolov reduction} functor, see \cref{sect::ds_reduction}. The finite Drinfeld-Sokolov reduction functor is a direct generalization of the definition of a W-algebra via quantum Hamiltonian reduction to the category of Harish-Chandra bimodules and refers to the quantum Drinfeld-Sokolov reduction, see \cite{FrenkelBenZvi}. Then, the subregular analog of the tensor isomorphisms \eqref{intro::tensor_structure} comes from the induced categorical action of the category of $\GL_N$-representations $\Rep(\GL_N)$ on the category of right $\cW$-modules, see \cref{thm::ds_reduction_tensor_structure_trivialization}. Unlike in the regular case of \cite{KalmykovYangians}, they depend on a dynamical-like parameters lying on a \emph{non-abelian} Lie subalgebra of $\gl_N$, and take the form
	\begin{equation}
		\label{intro::tensor_iso2}
		J_{UV} \colon U \otimes V \otimes \cW \rightarrow U \otimes V \otimes \cW.
	\end{equation}
	for $U,V$ representations of $\GL_N$.	
	
	The main results of the paper are contained in \cref{sect::tensor}. For instance, we provide an algorithm to compute the monoidal isomorphisms \eqref{intro::tensor_structure}. Finally, we compute its semi-classical limit by explicitly constructing the state-field correspondence for the vector representation $V=\C^{N}$ using the W-algebra generators from Brundan-Kleshchev, refer to \cref{thm:whittaker_geq2} and \cref{thm:whittaker_geq3}. Denote by $E_{i,j} \in \gl_N$ the matrix units. Consider the two-dimensional subalgebra $\l = \Span(E_{2,1},E_{1,1})$. Let $x_{21},x_{11}\in \l$ be functions on $\l^*$ corresponding to $E_{2,1},E_{1,1}$. The main result of the paper is \cref{thm::semiclassical_limit}.
	\begin{theorem}
		The semi-classical limit $\mathbf{j}$ of $J_{UV}$ from \eqref{intro::tensor_iso2} is
		\[
		\mathbf{j} = \mathbf{j}_c + \sum\limits_{j=2}^{N-2} \sum\limits_{i=j+2}^{N} \sum\limits_{r=2}^{i-j} (-1)^{i-j-r} x_{21} x_{11}^{i-j-r} E_{1,r} \otimes E_{i,j} + \sum\limits_{i=4}^{N} \sum\limits_{r=2}^{i-2} (-1)^{i-r} x_{11}^{i-r-1} E_{1,r} \otimes E_{i,1}.
		\]
	\end{theorem}
	
	Here $\mathbf{j}_c$ is constant, and it comes from a Frobenius-like structure given by the trace pairing with the subregular nilpotent element $e$ on a certain subspace of $\gl_N$ which we call the \emph{(subregular) wonderbolic} subspace:
	$$\w=\begin{pmatrix}
		0 & *&\ldots & *  & 0 \\
		0 & * & \ldots & * & 0 \\
		* & * & \ldots & * & 0 \\
		\vdots & \vdots &  \ddots & \vdots & \vdots \\
		* & * & \ldots & * & 0 \\
	\end{pmatrix}. $$
	See \cref{sect::semiclassical} for the explicit form for $\mathbf{j}_c$. The wonderbolic subspace plays a similar role to that of the \emph{mirabolic subalgebra} in the regular case, which was the main motivation for the name.
	
	It would be interesting to obtain an invariant description of the ``dynamical'' part similar to the constant one.
	
	\subsection*{Organization of the paper} In \cref{sect::background}, we introduce general notions, including that of a module category, that we will use in the paper. In \cref{section:pyramid}, we recall the Brundan-Kleshchev's construction \cite{BrundanKleshchev} of W-algebras for $\gl_N$ in terms of pyramids. In \cref{sect::ds_reduction}, we present the exchange construction in categorical terms as a finite Drinfeld-Sokolov reduction functor from the category of Harish-Chandra bimodules. In \cref{sect::subregular}, we specialize to the case of subregular W-algebras: in \cref{subsect::whittaker_vectors_standard_representation}, we compute Whittaker vectors for the vector representation of $\gl_N$, and in \cref{sect::tensor}, we compute the corresponding monoidal structure.
	
	\subsection*{Acknowledgments} This research was
	conducted while B.L. was a participant in the MIT PRIMES program; we would like to thank the organizers for making this research opportunity possible. A.K. would like to thank the Department of Mathematics of Massachusetts Institute of Technology for hospitality.
		
		\section{Background}
		\label{sect::background}
		
		\renewcommand{\theequation}{\arabic{section}.\arabic{equation}}
		\renewcommand{\U}{\mathrm{U}_{\hbar}}
		\renewcommand{\HC}{\mathrm{HC}_{\hbar}}
		
		\subsection{General setup}
		In this section, we introduce general notations that we use in the paper. 
		
		We work over the field of complex numbers $\C$. Throughout the paper, we use $\hbar$-versions of constructions in questions. To avoid categorical complications, we treat $\hbar$ as a \emph{non-formal} parameter, i.e. $\hbar \in \C^{\times}$ (for instance, we can still deal with $\C$-linear categories). The reader may safely assume that $\hbar=1$. The only purpose of introducing $\hbar$ is to compute classical limits of certain formulas, and it will be clear from the context how to make sense of the corresponding $\hbar$-family over $\mathbb{A}^1$.
		\newline
		
		Let $\g$ be a Lie algebra.
		\begin{defn}
			\label{def::asympt_uea}
			The \defterm{asymptotic universal enveloping algebra} $\U(\g)$ of $\g$ is a tensor algebra over $\C$, generated by the vector space $\g$, with the relations
			\[
			xy - yx = \hbar [x,y],\ x,y\in \g.
			\]
			For any $x,y\in \U(\g)$, the \defterm{commutator} $[x,y]$ is
			\begin{equation}
				\label{eq::asymptotic_commutator}
				[x,y] := \frac{xy-yx}{\hbar}.
			\end{equation}
			Observe that it is well-defined over $\C[\hbar]$.
		\end{defn}
		
		\begin{remark}
			Usually, asymptotic universal enveloping algebras are defined over the polynomial ring $\C[\hbar]$. Here, as we mentioned at the beginning of the section, we treat $\hbar$ just as a complex number.
		\end{remark}
		
		In this paper, we will be dealing with a general linear Lie algebra. Let $\GL_N$ be the group of invertible $N \times N$-matrices and $\gl_N$ be its Lie algebra identified with the space of $N\times N$-matrices. We choose a natural basis $\{E_{i,j} | 1 \leq i,j \leq N\}$ of matrix units with the commutator
		\[
		[E_{i,j}, E_{k,l}] = \delta_{jk} E_{i,l} - \delta_{li} E_{k,j}.
		\]
		
		\subsection{Module categories}
		\label{subsect::module_categories}
		In this subsection, we recall the notion of a module category over a tensor category, for instance, see \cite[Chapter 7]{EGNO}.
		
		Recall that a monoidal category is a plain category $\cC$ equipped with a bifunctor $\otimes \colon \cC \times \cC \rightarrow \cC$ together with a unit object $\mathbf{1}_{\cC}$ and a natural isomorphism (associativity constraint)
		\[
		a_{X,Y,Z} \colon (X \otimes Y) \otimes Z \xrightarrow{\sim} X \otimes (Y \otimes Z),
		\]
		satisfying the \emph{unit} and \emph{pentagon} axioms, see \cite[Chapter 2]{EGNO}. Monoidal categories are categorical analogs of algebras. Likewise, there is a categorical analog of modules over an algebra.
		\begin{defn}
			\label{def::module_categories}
			Let $\cC$ be a monoidal category. A \defterm{(right) module category} over $\cC$ is a plain category $\cM$ equipped with a bifunctor
			\[
			\otimes \colon \cM \times \cC \rightarrow \cM
			\]
			with natural isomorphisms
			\[
			m_{M,X,Y} \colon M \otimes (X \otimes Y) \xrightarrow{\sim} (M \otimes X) \otimes Y,
			\]
			for all $M\in \cM,\ X,Y\in \cC$, such that the functor $M \otimes \mathbf{1}_{\cC} \mapsto M$ is an autoequivalence of $\cM$ and the associativity constraint $m$ satisfies the \emph{pentagon axiom}, shown below.
			\[
			\xymatrix{
				& M \otimes (X \otimes (Y \otimes Z) )\ar[dr]^{\id_M \otimes a_{X,Y,Z}^{-1}} \ar[dl]_{m_{M,X,Y\otimes Z}} \\
				(M \otimes X) \otimes (Y \otimes Z) \ar[d]_{m_{M\otimes X,Y,Z}} & & M \otimes ((X \otimes Y) \otimes Z) \ar[d]^{m_{M,X\otimes Y,Z}} \\
				((M \otimes X) \otimes Y) \otimes Z \ar[rr]^{m_{M,X,Y} \otimes \id_Z} & & (M \otimes (X \otimes Y)) \otimes Z
			}.
			\]
		\end{defn}
		
		\begin{example}
			Tautologically, any monoidal category $\cC$ is a module category over itself.
		\end{example}
		
		There is also a generalization of a homomorphism between algebra modules.
		\begin{defn}
			\label{def::module_categories_functor}
			Let $\cC$ be a monoidal category and $\cM_1,\cM_2$ be two module categories over $\cC$. A \defterm{functor of $\cC$-module categories} is a plain functor $F\colon \cM_1 \rightarrow \cM_2$ with a collection of natural isomorphisms
			\begin{equation}
				\label{eq::module_categories_functor}
				J_{M,X} \colon F(M \otimes X) \rightarrow F(M) \otimes X
			\end{equation}
			for all $M\in \cM,\ X\in \cC$, satisfying the compatibility condition
			\[
			\xymatrix{
				& F(M \otimes (X\otimes Y)) \ar[dr]^{F(m_{M,X,Y})} \ar[dl]_{J_{M,X\otimes Y}} & \\
				F((M \otimes X) \otimes Y) \ar[d]_{J_{M\otimes X,Y}} & & F(M) \otimes (X \otimes Y) \ar[d]^{n_{F(M),X,Y}} \\
				F(M\otimes X) \otimes Y \ar[rr]^{J_{M,X} \otimes \id_Y} & & (F(M) \otimes X) \otimes Y
			}
			\]
			and 
			\[
			\xymatrix{
				F(M \otimes \mathbf{1}_{\cC}) \ar[dr] \ar[rr]^{J_{M,\mathbf{1}_{\cC}}} & & F(M) \otimes \mathbf{1}_{\cC} \ar[dl] \\
				& F(M) & 
			}
			\]
			where the diagonal arrows come from the corresponding unit autoequivalences.
		\end{defn}

		\section{W-algebras for $\gl_N$}
		\label{section:pyramid}
		
		In general, a \emph{finite W-algebra} is associated to a pair $(\g,e)$ of a reductive Lie algebra $\g$ and a nilpotent element $e\in \g$, see \cite{Losev} for a survey of the subject. From now on, we will be interested in the case $\g = \gl_N$; according to \cite{BrundanKleshchev}, the W-algebras in this case admit a description in terms of combinatorial objects called \emph{pyramids}. In what follows, we recall this description; for the details and proofs, we refer the reader to \emph{loc. cit.}
		
		\begin{defn}\cite[Section 7]{BrundanKleshchev}
			\label{def::pyramid}
			A \defterm{pyramid} $\pi$ is a sequence of positive numbers $(q_1,\ldots,q_l)$, the \emph{column heights}, such that $\sum\limits_{i=1}^l q_i = N$ and 
			\[
			0 < q_1  \leq \dots \leq q_k, \quad q_{k+1} \geq \dots \geq q_l > 0
			\]
			for some $k \leq l$. The \defterm{maximal height} $n$ is $\max(q_1,\ldots,q_l)$.
		\end{defn}
		We number the blocks starting from top to bottom and from left to right; for each $i$, we denote by $\col(i)$ (resp. $\row(i)$) the corresponding column (resp. row) of $i$ counted from left to right (resp. from top to bottom). Here is an example of a pyramid
		\begin{equation*}
			\label{eq::pyramid}
			\begin{picture}(65, 45)%
				\put(0,0){\line(1,0){60}}
				\put(0,15){\line(1,0){60}}
				\put(15,30){\line(1,0){30}}
				\put(15,45){\line(1,0){15}}
				\put(0,0){\line(0,1){15}}
				\put(15,0){\line(0,1){45}}
				\put(30,0){\line(0,1){45}}
				\put(45,0){\line(0,1){30}}
				\put(60,0){\line(0,1){15}}
				\put(7,8){\makebox(0,0){1}}
				\put(22,8){\makebox(0,0){4}}
				\put(22,38){\makebox(0,0){2}}
				\put(22,23){\makebox(0,0){3}}
				\put(37,8){\makebox(0,0){6}}
				\put(37,23){\makebox(0,0){5}}
				\put(52,8){\makebox(0,0){7}}
				\put(-20,18){\makebox(0,0){$\pi =$}}
			\end{picture}
		\end{equation*}
		with column heights $(q_1,q_2,q_3,q_4)=(1,3,2,1)$. For instance, $\col(3) = 2$ and $\row(3) = 2$. Observe the rows must be in non-decreasing order.
		
		In what follows, we will need an inductive structure on pyramids.
		\begin{defn}
			\label{def::pyramid_truncation}
			For a pyramid $\pi$, the \defterm{k-th truncated pyramid} ${}_k \pi$ is $\pi$ without the last $k$ rightmost columns. We denote by ${}_k N$ the number of blocks in ${}_k \pi$. When $k=1$, we also denote $\dot{\pi} := {}_1 \pi$.
		\end{defn} 
		
		For instance, we have
		\[
		\begin{picture}(65, 45)%
			\put(0,0){\line(1,0){45}}
			\put(0,15){\line(1,0){45}}
			\put(15,30){\line(1,0){30}}
			\put(15,45){\line(1,0){15}}
			\put(0,0){\line(0,1){15}}
			\put(15,0){\line(0,1){45}}
			\put(30,0){\line(0,1){45}}
			\put(45,0){\line(0,1){30}}
			\put(7,8){\makebox(0,0){1}}
			\put(22,8){\makebox(0,0){4}}
			\put(22,38){\makebox(0,0){2}}
			\put(22,23){\makebox(0,0){3}}
			\put(37,8){\makebox(0,0){6}}
			\put(37,23){\makebox(0,0){5}}
			\put(-20,18){\makebox(0,0){$_1\pi =$}}
		\end{picture}
		\]
		for $\pi$ as above. Introduce a $\Z$-grading on $\g = \bigoplus_{j\in \Z} \g_j$ by declaring that $\deg(E_{ij}) = \col(j) - \col(i)$. Let 
		\begin{equation}
			\label{eq::pyramid_nilpotent_subalgebra}
			\p = \bigoplus\limits_{j \geq 0} \g_j, \qquad \m = \bigoplus\limits_{j < 0} \g_j.
		\end{equation}
		
		Similarly to \cref{def::pyramid_truncation}, we will use an inductive structure on the algebras.
		\begin{defn}
			\label{def::subalgebras_truncation}
			The \defterm{k-th truncated nilpotent subalgebra} ${}_k\m$ (resp. \defterm{k-th truncated parabolic subalgebra}) is the nilpotent subalgebra (resp. the parabolic subalgebra) associated to the truncated pyramid ${}_k \pi$. Alternatively, we denote them by $\m_{{}_k N}$ (resp. $\p_{{}_k N}$) if the truncation is clear from the context.
		\end{defn}
		
		To a pyramid $\pi$, we assign a nilpotent element $e$ defined by
		\begin{equation}
			\label{eq::nilpotent}
			e:= \sum\limits_{\substack{1 \leq i,j \leq N \\ \row(i) = \row(j) \\ \col(i) = \col(j) -1}} E_{i,j}.
		\end{equation}
		For instance, for pyramid \eqref{eq::pyramid}, we have $e = E_{3,5} + E_{1,4} + E_{4,6} + E_{6,7}$. 
		
		The nilpotent element $e$ defines a character $\psi$ of $\m$:
		\[
		\psi \colon \m \rightarrow \C, \qquad x \mapsto \Tr(ex).
		\]
		Denote by 
		\begin{equation}
			\label{eq::whittaker_module}
			Q := \U(\g) \otimes_{\U(\m)} \C^{\psi},
		\end{equation}
		where $\m$ acts on $\C^{\psi}$ via the character $\psi$. It is naturally a left $\U(\g)$-module. As a vector space, we can identify 
		\begin{equation}
			\label{eq::quotient_identification}
			Q \cong \U(\p)
		\end{equation}
		by the PBW theorem.
		
		For any $\xi \in \m$, denote by $\xi^{\psi} = \xi - \psi(\xi)$, and define the shift
		\begin{equation}
			\label{eq::nilpotent_shift}
			\m^{\psi} = \Span(\xi^{\psi}| \xi \in \m) \subset \U(\g).
		\end{equation}
		\begin{defn}\cite[Section 8]{BrundanKleshchev}
			A \defterm{finite W-algebra} $\cW$, associated to the nilpotent element $e$ \eqref{eq::nilpotent}, is the space 
			\[
			\cW := Q^{\m^{\psi}} = \{ w \in Q| \xi^{\psi} w = 0\ \forall \xi \in \m\}
			\]
			of $\m^{\psi}$-invariant vectors in $Q$.
		\end{defn}
		
		In \cite{BrundanKleshchev}, the authors introduced explicit generators of $\cW$ whose construction we recall now. Observe that in \emph{loc. cit.}, the authors use the version with $\hbar=1$; to pass to the ``asymptotic'' version, one can use the Rees construction with respect to the \emph{Kazhdan filtration}, see \cite[(8.3)]{BrundanKleshchev}. We will explicitly indicate how to modify the corresponding definitions and statements.
		
		Let $\rho_{\pi,r}=n-\sum\limits_{k=r}^{l}q_k$. Introduce modified generators
		\begin{equation}
			\label{eq::modified_generators}
			\widetilde{E}_{i,j}=(-1)^{\mathrm{col}(j)-\mathrm{col}(i)}(E_{i,j}+\delta_{ij} \hbar \rho_{\pi,\mathrm{col}(i)})
		\end{equation} 
		for all $1 \le i,j \le n$.
		\begin{defn}
			\label{def::kazhdan_filtration}
			The \defterm{Kazhdan filtration} on $\U(\g)$ is defined by declaring that $\deg(E_{i,j}) = \col(j) - \col(i) + 1$. 
		\end{defn}
		In particular, assigning $\deg(\hbar)=1$, we see that this modification preserves the filtration. 
		\newline
		\newline
		
		Let $1 \leq x \leq n$. 
		\begin{defn}\cite[Section 9]{BrundanKleshchev}
			\label{def:t_els}
			Consider the set of signs $\sigma_1=\sigma_2=\ldots=\sigma_x=-, \sigma_{x+1}=\ldots=\sigma_{n}=+$. For any $1 \leq i,j \leq N$, define $T_{ij,x}^{(0)} := \delta_{ij} \sigma_i$ and, for $r>0$,
			\begin{equation}
				\label{eq::bk_t_elements}
				T_{ij;x}^{(r)}=\sum_{s=1}^{r}\sum\limits_{\substack{i_1,\ldots,i_s \\ j_1,\ldots,j_s}}^{}\sigma_{\mathrm{row}(j_1)}\cdots \sigma_{\mathrm{row}(j_{s-1})}\widetilde{E}_{i_1,j_1}\cdots \widetilde{E}_{i_s,j_s},
			\end{equation}
			with the following conditions on $1 \le i_1,\ldots,i_s,j_1,\ldots,j_s \le N$:
			\begin{enumerate}
				\item \label{item::first} $\mathrm{row}(i_1)=i, \mathrm{row}(j_s)=j$
				\item \label{item::second} $\mathrm{row}(j_k)=\mathrm{row}(i_{k+1}) \text{ for all } 1 \le k \le s-1$
				\item \label{item::bk_t_parabolic} $\mathrm{col}(i_k) \le \mathrm{col}(j_k) \text{ for all } 1 \le k \le s$
				\item \label{item::degree} $\sum\limits_{k=1}^{s}(\mathrm{col}(j_k)-\mathrm{col}(i_k)+1)=r$
				\item \label{item::fifth} If $\sigma_{\mathrm{row}(j_k)}=+$, then $\mathrm{col}(j_k) < \mathrm{col}(i_{k+1}) \text{ for all } 1 \le k \le s-1$.
				\item \label{item::sixth} If $\sigma_{\mathrm{row}(j_k)}=-$, then $\mathrm{col}(j_k) \ge \mathrm{col}(i_{k+1}) \text{ for all } 1 \le k \le s-1$. 
			\end{enumerate}
		\end{defn}
		
		It follows from condition \eqref{item::bk_t_parabolic} that $T_{ij;x}^{(r)} \in \U(\p)$ by \eqref{eq::quotient_identification}. Condition \eqref{item::degree} can be equivalently reformulated that the degree of $T_{ij;x}^{(r)}$ with respect to the Kazhdan filtration is $r$. For instance, when $r=1$, it can be shown that
		\begin{equation}
			\label{eq::bk_generator_degree_one}
			T_{ij;x}^{(1)}=\sum\limits_{\substack{1 \le h,k \le n \\ \mathrm{col}(h)=\mathrm{col}(k) \\ \mathrm{row}(h)=i, \mathrm{row}(k)=j}} \widetilde{E}_{h,k}.
		\end{equation}
		
		Using these elements, the authors of \cite{BrundanKleshchev} constructed generators of a W-algebra corresponding to a pyramid $\pi$. We will give a precise statement for the subregular case in \cref{sect::subregular}.
		
		We will also need an inductive structure on these elements. Recall \cref{def::pyramid_truncation}. Let $\gl_{{}_k N}$ be the Lie algebra corresponding to the truncated pyramid ${}_k\pi$. Denote by ${}_k \widetilde{E}_{i,j}$ the corresponding modified generators \eqref{eq::modified_generators}. Consider a (non-standard) embedding
		\begin{equation}
			\label{eq::nonstandard_embedding}
			\iota \colon \U(\gl_{{}_k N}) \rightarrow \U(\gl_N),\ \iota({}_k \widetilde{E}_{i,j}) =  \widetilde{E}_{i,j}.
		\end{equation}
		Define the truncated analog of elements \eqref{eq::bk_t_elements} as
		\begin{equation}
			\label{eq::bk_t_elements_truncated}
			{}_k T_{ij;x}^{(r)} := \iota(T_{ij;x}^{(r)}).
		\end{equation}
		
		\section{Finite Drinfeld-Sokolov reduction}
		\label{sect::ds_reduction}

		\subsection{Harish-Chandra bimodules}
		\label{subsect::hc_bimodules}
		Let $G$ be an affine algebraic group over $\C$ and $\g$ be its Lie algebra. Denote by $\U(\g)$ the universal enveloping algebra of $\g$ as in \cref{def::asympt_uea}. Let $\Rep(G)$ be the category of $G$-representations. Naturally, $\U(\g)$ is an object in $\Rep(G)$. 
		
		\begin{defn}
			A \defterm{Harish-Chandra bimodule} is a left $\U(\g)$-module $X$ in the category $\Rep(G)$. In other words, it has a structure of a $G$-representation and a left $\U(\g)$-module such that the action morphism
			\[
			\U(\g) \otimes X \rightarrow X
			\]
			is a homomorphism of $G$-representations. The category of Harish-Chandra bimodules is denoted by $\HC(G)$.
		\end{defn}
		
		There is a natural right $\U(\g)$-module structure on any Harish-Chandra bimodule $X$ (justifying the name). Namely, for $\xi \in\g$, denote by $\ad_{\xi}\colon X \rightarrow X$ the derivative of the $G$-action on $X$ along $\xi$. Then we can define
		\begin{equation}
			\label{eq::hc_right_action}
			x\xi := \xi x - \hbar \ad_{\xi}(x),\ x\in X,
		\end{equation}
		and extend it to a right $\U(\g)$-action. Therefore, the category $\HC(G)$ is a subcategory of $\U(\g)$-bimodules, hence is equipped with a tensor structure:
		\[
		X \otimes^{\HC(G)} Y := X \otimes_{\U(\g)} Y.
		\]
		
		There is a natural functor of the so-called \defterm{free Harish-Chandra bimodules}:
		\begin{equation}
			\label{eq::free_hc_bimodules}
			\free \colon \Rep(G) \rightarrow \HC(G), \qquad V \mapsto \U(\g) \otimes V.
		\end{equation}
		One can check that this functor is monoidal. In fact, all Harish-Chandra bimodules can be ``constructed'' from the free ones.
		\begin{prop}\cite[Proposition 2.7]{KalmykovSafronov}
			\label{prop::hc_generators}
			The category $\HC(G)$ is generated by $\free(V)$ for $V \in \Rep(G)$.
		\end{prop}
		
		\subsection{Drinfeld-Sokolov reduction}
		\label{subsect::ds_reduction}
		Now let restrict to the case $G = \GL_N$ and $\g = \gl_N$. We use the notations from \cref{section:pyramid}, in particular, we fix a pyramid $\pi$ and consider the corresponding nilpotent subalgebra $\m$ with a character $\psi\in \m^*$.
		\begin{defn}
			A \defterm{Whittaker module} is a left $\U(\g)$-module $M$ such that the action of $\m^{\psi}$ from \eqref{eq::nilpotent_shift} is locally nilpotent. A \defterm{Whittaker vector} is an $\m^{\psi}$-invariant vector in $m\in M$, i.e. satisfying
			\[
			\xi^{\psi} m = 0\ \text{ for all } \xi \in \m.
			\]
			The space of Whittaker vectors is denoted by $M^{\m^{\psi}}$.
		\end{defn} 
		
		\begin{example}
			In $\mathfrak{gl}_2$, the series $$P^{\psi}=\sum_{k=0}^{\infty}(-1)^{k}\frac{E_{1,1}(E_{1,1}+1)\cdots(E_{1,1}+k-1)}{k!}(E_{2,1}-1)^k$$ is $(E_{21}-1)$-invariant on the left action and generates the Whittaker vectors; see \cite{Kalmykov} for a version in the left quotient.
		\end{example}
		
		Denote by $\Wh$ the category of $(\U(\g),\cW)$-bimodules that are Whittaker with respect to the $\U(\g)$-action. Naturally, the quotient $Q$ from \eqref{eq::whittaker_module} is an object of $\Wh$. In particular, it defines an action functor
		\begin{equation*}
			\act_{\g}^{\psi} \colon \HC(G) \rightarrow \Wh, \qquad X \mapsto X \otimes_{\U(\g)} Q.
		\end{equation*}
		Likewise, there is an action 
		\begin{equation}
			\act_{\cW} \colon \BiMod{\cW}{\cW} \rightarrow \Wh, \qquad Y \mapsto Q \otimes_{\cW} Y.
		\end{equation}
		Consider the functor
		\[
		(-)^{\m^{\psi}} \colon \Wh \rightarrow \BiMod{\cW}{\cW}
		\]
		of Whittaker invariants sending a Whittaker module to its space of Whittaker vectors. 
		
		The following result is a direct consequence of \emph{Skryabin's equivalence} \cite{Premet}.
		\begin{thm}
			The functor $(-)^{\m^{\psi}}$ is an equivalence.
		\end{thm}
		This motivates the following definition.
		\begin{defn}
			\label{def::ds_reduction}
			The \defterm{(finite) Drinfeld-Sokolov reduction} is the functor
			\[
			\res^{\psi}\colon \HC(\GL_N) \rightarrow \BiMod{\cW}{\cW}, \qquad X \mapsto (X \otimes_{\U(\g)} Q)^{\m^{\psi}}.
			\]
		\end{defn}
		In what follows, for any Harish-Chandra bimodule $X$, we denote
		\[
		X /\m^{\psi} := X \otimes_{\U(\g)} Q.
		\]
		\begin{remark}
			\label{rmk::whittaker_adjoint}
			There is an equivalent presentation of the Drinfeld-Sokolov reduction that we will use later in the paper. Namely, recall the adjoint $\gl_N$-action from \cref{subsect::hc_bimodules}. For any Harish-Chandra bimodule $X$, define
			\[
			\ad_m([x]) := [\ad_m(x)]\in X/\m^{\psi},\ m\in \m, [x] \in X/\m^{\psi}.
			\]
			We will also use the notation
			\[
			[m,x] := \ad_m([x]),\ m\in \m, x\in X,
			\]
			if the quotient is clear from the context. Since $\psi$ is a character, this action is well-defined. One can easily see that 
			\[
			\hbar\cdot\ad_m([x]) = m^{\psi}\cdot [x].
			\]
			In particular, the space of Whittaker vectors in $X/\m^{\psi}$ can be identified with the space of $\ad_{\m}$-invariant vectors in $X/\m^{\psi}$.
		\end{remark}
		
		As in \cite[Corollary 4.18]{KalmykovSafronov}, we obtain the following.
		\begin{thm}
			\label{thm::ds_reduction_monoidal}
			The Drinfeld-Sokolov reduction is colimit-preserving and monoidal.
		\end{thm}
		Explicitly, the monoidal structure is given by the usual product on quantum Hamiltonian reductions:
		\begin{equation}
			\label{eq::explicit_tensor_structure}
			(X/\m^{\psi})^{\m^{\psi}} \otimes_{\cW} (Y/\m^{\psi})^{\m^{\psi}} \xrightarrow{\sim} (X \otimes_{\U(\g)} Y / \m^{\psi})^{\m^{\psi}},\ [x] \otimes [y] \mapsto [x \otimes y].
		\end{equation}
		
		In particular, composing with the monoidal functor of free Harish-Chandra bimodules \eqref{eq::free_hc_bimodules}, we get a monoidal functor:
		\begin{equation}
			\label{eq::ds_reduction_funtor_rep_g}
			\Rep(G) \rightarrow \BiMod{\cW}{\cW}.
		\end{equation}
		We study its properties in the next section.

		\section{Subregular case}
		\label{sect::subregular}
		In this section, we apply the finite Drinfeld-Sokolov reduction to subregular W-algebras and study its tensor properties.
		
		\subsection{Pyramid}
		Recall from \cref{section:pyramid} that W-algebras for $\gl_N$ are described by pyramids. In the subregular case, the corresponding pyramid is 
		\begin{equation}
			\label{eq::subregular_pyramid}
			\begin{picture}(150, 30)%
				\put(0,0){\line(1,0){60}}
				\put(0,15){\line(1,0){60}}
				\put(0,30){\line(1,0){30}}
				\put(90,0){\line(1,0){60}}
				\put(90,15){\line(1,0){60}}
				\put(0,0){\line(0,1){30}}
				\put(30,0){\line(0,1){30}}
				\put(60,0){\line(0,1){15}}
				\put(90,0){\line(0,1){15}}
				\put(120,0){\line(0,1){15}}
				\put(150,0){\line(0,1){15}}
				\put(15,23){\makebox(0,0){1}}
				\put(15,8){\makebox(0,0){2}}
				\put(45,8){\makebox(0,0){3}}
				\put(105,8){\makebox(0,0){$N-1$}}
				\put(135,8){\makebox(0,0){$N$}}
				\put(75,4){\makebox(0,0){\ldots}}
				\put(-20,12){\makebox(0,0){$\pi =$}}
			\end{picture}
		\end{equation}
		and by \eqref{eq::nilpotent}, the subregular nilpotent is given by 
		\begin{equation}
			\label{eq::subregular_nilpotent}
			e = E_{2,3} + \ldots + E_{N-1,N}.
		\end{equation}
		The nilpotent algebra $\m$ is 
		\[\m = \Span (E_{i,j} | 3 \leq i \leq N, j < i), \]
		namely, 
		\begin{equation}
			\label{eq::subregular_nilpotent_subalgebra}
			\m = \begin{pmatrix}
				0 & 0& 0 &\ldots & 0  & 0 \\
				0 & 0 &  0 &\ldots & 0 & 0 \\
				* & * & 0 & \ldots & 0 & 0 \\
				* & * & * & \ldots & 0 & 0 \\
				\vdots & \vdots & \vdots&  \ddots & \vdots & \vdots \\
				* & * & * &\ldots & * & 0 \\
			\end{pmatrix}.
		\end{equation}
		The parabolic subalgebra $\p$ is 
		\begin{equation}
			\label{eq::subregular_parabolic_subalgebra}
			\p = \begin{pmatrix}
				* & *& * &\ldots & *  & * \\
				* & * &  * &\ldots & * & * \\
				0 & 0 & * & \ldots & * & * \\
				0 & 0 & 0 & \ldots & * & * \\
				\vdots & \vdots & \vdots&  \ddots & \vdots & \vdots \\
				0 & 0 & 0 &\ldots & 0 & * \\
			\end{pmatrix}.
		\end{equation}
		
		\subsection{Semi-classical limit}
		\label{sect::semiclassical}
		It turns out that the semi-classical limits of the tensor structure on \eqref{eq::ds_reduction_funtor_rep_g} is intrinsically related not to the whole Lie algebra $\gl_N$, but to its subspace of a certain almost parabolic subalgebra.
		\begin{defn}
			The \defterm{subregular wonderbolic subspace} $\w$ (for the rest of the paper, simply \defterm{wonderbolic subspace}) is the subspace of matrices of the form
			\[
			\w = \begin{pmatrix}
				0 & *&\ldots & *  & 0 \\
				0 & * & \ldots & * & 0 \\
				* & * & \ldots & * & 0 \\
				\vdots & \vdots &  \ddots & \vdots & \vdots \\
				* & * & \ldots & * & 0 \\
			\end{pmatrix}.
			\]
		\end{defn}
		
		While the subregular nilpotent element $e$ from \eqref{eq::subregular_nilpotent} does not lie in $\w$, it defines the following 2-form on $\w$:
		\[
		\omega\colon \w \wedge \w \rightarrow \C,\ x \wedge y \mapsto \Tr(e\cdot [x,y]).
		\]
		
		Observe that the nilpotent subalgebra $\m$ from \eqref{eq::subregular_nilpotent_subalgebra} lies in $\w$, moreover, it is isotropic with respect to $\omega$. A natural complement is given by the Borel subalgebra
		\[
		\b = \Span (E_{k,l} | 1 \leq k \leq l \leq N-1, 2 \leq l).
		\]
		Namely, 
		\begin{equation}
			\label{eq::wonderbolic_borel}
			\b = \begin{pmatrix}
				0 & *& * &\ldots & *  & 0 \\
				0 & * & * & \ldots & * & 0 \\
				0 & 0 & * & \ldots & * & 0 \\
				\vdots & \vdots & \vdots &   \ddots & \vdots & \vdots \\
				0 & 0 & 0 & \ldots & * & 0 \\
				0 & 0 & 0 & \ldots & 0 & 0 \\
			\end{pmatrix}.
		\end{equation}
		Similarly to \cref{def::subalgebras_truncation}, we will use the following.
		\begin{defn}
			\label{def::borel_truncation}
			The \defterm{k-th truncated Borel subalgebra} ${}_k \b$ is the Borel subalgebra as in \eqref{eq::wonderbolic_borel} associated to the truncated pyramid ${}_k \pi$ from \cref{def::pyramid_truncation}. Alternatively, we denote it by $\b_{{}_k N}$, if the truncation is clear from the context.
		\end{defn}
		
		It turns out that $\omega$ is symplectic and both spaces are Lagrangian.
		\begin{prop}
			\label{prop::wonderbolic_r_matrix}
			The form $\omega$ is non-degenerate with inverse $r_{\w}=\mathbf{j}_c - \mathbf{j}_c^{21}$, where $\mathbf{j}_c^{21}$ is uniquely defined by 
			\begin{equation}
				\label{eq::subregular_r_matrix_conditions}
				\mathbf{j}_c^{21}(E_{i,j}^*) = \begin{cases}
					\delta_{i>2}E_{j,i-1},\ j=1,2 \\
					E_{j,i-1} - \mathbf{j}_c^{21}(E_{i-1,j-1}^*),\ j\geq 3
				\end{cases}
			\end{equation}
			for any $E_{i,j}^* \in \m^*$, where $E_{i,j}^*$ is the dual basis and we consider $\mathbf{j}_c^{21}$ as a map $\m^* \rightarrow \b$.
			\begin{proof}
				Observe that both $\m$ and $\b$ are isotropic subspaces with respect to $\omega$. Since the latter is skew-symmetric, it is enough to construct an inverse $-\mathbf{j}_c^{21}$ only of one map, say $\omega \colon \b \rightarrow \m$. Observe that
				\[
				\omega(E_{j,i-1}) = \begin{cases}
					-\delta_{i>2}E_{i,j}^*,\ j=1,2 \\
					-E_{i,j}^* + E_{i-1,j-1}^*,\ j \geq 3
				\end{cases}
				\]
				for $i \geq j+1$. Then \eqref{eq::subregular_r_matrix_conditions} follows. Note that these equations allow to construct $\mathbf{j}_c^{21}$ inductively, starting from $j=1$ and $j=2$. In particular, they define the inverse. 
			\end{proof}
		\end{prop}

		\begin{remark}
			The subregular wonderbolic subspace is an analog of the \emph{mirabolic subalgebra} in the case of the regular nilpotent element in the same way $r_{\w}$ is an analog of the rational Cremmer-Gervais $r$-matrix, see \cite{KalmykovYangians}. One main difference is that it is not a subalgebra, thus $r_{\cW}$ does not satisfy the classical Yang-Baxter equation. However, it turns out $\mathbf{j}_c$ is the constant part of the semi-classical limit of the tensor structure on Whittaker vectors. As the reader will see in \cref{sect::tensor}, in addition to the constant part $r_{\w}$, the semi-classical limit of the tensor structure also involves certain ``dynamical'' parameters lying on the subalgebra spanned by $\{E_{11},E_{21}\}$.
		\end{remark}

		\subsection{Whittaker vectors: general setup} 
		\label{subsect::whittaker_vectors_general_setup}
		In this subsection, we show that the Drinfeld-Sokolov reduction functor \eqref{eq::ds_reduction_funtor_rep_g} admits a canonical ``trivialization.''
		
		Recall the elements $T_{ij;x}^{(r)}$ from \eqref{eq::bk_t_elements} and their truncated analogs ${}_k T_{ij;x}^{(r)}$ from \eqref{eq::bk_t_elements_truncated}. As we mentioned in \cref{section:pyramid}, the authors of \cite{BrundanKleshchev} considered the case $\hbar=1$, however, all the proofs can be translated \emph{mutatis mutandis} to their $\hbar$-versions and will not be mentioned explicitly here.
		
		Recall also that a W-algebra is defined as the quantum Hamiltonian reduction $(\U(\g)/\m^{\psi})^{\m^{\psi}}$. Since $T_{ij;x}^{(r)} \in \U(\p)$ by construction, we may treat them as elements in the quotient $\U(\g)/\m^{\psi}$. By combination of a particular case of the fundamental result \cite[Theorem 10.1]{BrundanKleshchev}, identifying W-algebras with \emph{truncated shifted Yangians}, and \cite[Corollary 6.3]{BrundanKleshchev}, subregular W-algebras admit an explicit presentation.
		\begin{thm}
			\label{thm::subregular_walgebra_basis}
			The monomials in the elements
			\begin{equation*}
				T_{11;0}^{(1)}, \qquad T_{21;1}^{(1)}, \qquad T_{12;1}^{(N-1)}, \qquad \{T_{22;1}^{(r)} \}_{1 \leq r \leq N-1},
			\end{equation*}
			taken in any fixed order, form a basis of the subregular W-algebra $\cW$.
		\end{thm}
		
		Observe that $T_{11;0}^{(1)} = E_{1,1} - (N-2) \hbar$ and $T_{21;1}^{(1)}=-E_{2,1}$. We will consider the following subalgebra of $\p$:
		\begin{equation}
			\label{eq::l_subalgebra}
			\l := \Span(E_{2,1},E_{1,1}).
		\end{equation}
		
		For any left $\U(\g)$-module $X$, denote by $\b \backslash X := \C \otimes_{\U(\b)} \U(\g)$, where $\b$ acts on $\C$ trivially. Consider the composition
		\begin{equation}
			\label{eq::subregular_walgebra_quotient}
			\cW = (\U(\g)/\m^{\psi})^{\m^{\psi}} \hookrightarrow \U(\g)/\m^{\psi} \rightarrow \b \backslash \U(\g) / \m^{\psi}.
		\end{equation}
		\begin{prop}
			\label{prop::subregular_walgebra_quotient}
			The map \eqref{eq::subregular_walgebra_quotient} is an isomorphism of right $\cW$-modules.
			\begin{proof}
				Consider the map between the associated graded spaces with respect to the filtration induced from the Kazhdan grading on both sides (recall that $\hbar \in \C^*$). It is clear from the formula \eqref{eq::bk_t_elements} that it sends
				\[
				T_{11;0}^{(1)} \mapsto E_{1,1}, \qquad T_{21;1}^{(1)} \mapsto E_{2,1}, \qquad T_{22;1}^{(N-1)} \mapsto E_{N,N}.
				\]
				It is also clear that it sends 
				\[
				T_{22;1}^{(r)} \mapsto E_{r+1,N} + x,\ 1\leq r < N-1,
				\]
				where $x$ is expressible in terms of $E_{11},E_{21},E_{s+1,N}$ for $s>r$. Likewise, 
				\[
				T_{12;1}^{(N-1)} = E_{1,N} + x,
				\]
				where $x$ is expressible in terms of $E_{1,1},E_{2,1},E_{s+1,N}$ for $s\geq 1$. In particular, we see that this map sends generators to generators. Since this is an algebra homomorphism, we conclude by \cref{thm::subregular_walgebra_basis} that it is an isomorphism. In particular, the map \eqref{eq::subregular_walgebra_quotient} is an isomorphism as well.
			\end{proof}
		\end{prop}
		
		Recall the setting of \cref{sect::ds_reduction}.
		\begin{cor}
			\label{cor::subregular_res_quotient}
			For any Harish-Chandra bimodule $X$, there is a natural isomorphism of right $\cW$-modules
			\[
			\res^{\psi}(X) = (X/\m^{\psi})^{\m^{\psi}} \xrightarrow{\sim} \b \backslash X /\m^{\psi}.
			\]
			\begin{proof}
				Recall that by Skryabin's theorem, the natural action map
				\[
				\U(\g)/\m^{\psi} \otimes_{\cW} (X/\m^{\psi})^{\m^{\psi}} \rightarrow X /\m^{\psi}
				\]
				is an isomorphism. Therefore, by \cref{prop::subregular_walgebra_quotient} we have
				\[
				\b \backslash X / \m^{\psi} \cong \b \backslash \U(\g) /\m^{\psi} \otimes_{\cW} (X/\m^{\psi})^{\m^{\psi}} \cong \cW \otimes_{\cW} (X/\m^{\psi})^{\m^{\psi}} = (X/\m^{\psi})^{\m^{\psi}}
				\]
				as required.
			\end{proof}
		\end{cor}
		
		In particular, it implies that we can ``trivialize'' the Drinfeld-Sokolov reduction \eqref{eq::ds_reduction_funtor_rep_g} on free Harish-Chandra bimodules.
		\begin{prop}
			\label{cor::ds_trivialization}
			For any $V\in \Rep(G)$, there is a natural isomorphism of right $\cW$-modules
			\begin{equation}
				\label{eq::categorical_trivialization}
				\triv_V \colon V \otimes \cW \xrightarrow{\sim} \b \backslash \U(\g) \otimes V /\m^{\psi},
			\end{equation}
			i.e. for every $v\in V$, there exists a unique Whittaker vector $v^{\psi}$ of the form
			\begin{equation}
				\label{eq::whittaker_vector_canonical_form}
				v^{\psi} = 1 \otimes v + \sum x_i \otimes v_i,\ x_i \in \b \cdot \U(\g).
			\end{equation}
			In particular, together with the isomorphism
			\[
			\b \backslash \U(\g) / \m^{\psi} \rightarrow \res^{\psi}(\U(\g) \otimes V),
			\]
			we have a commutative diagram
			\[
			\xymatrix@C+1.5pc{
				\Rep(G) \ar[dr]_{\free_{\cW}} \ar[r]^-{\res^{\psi} \circ \free } & \BiMod{\cW}{\cW} \ar[d]^{\forget} \\ & \RMod_{\cW}
			}
			\]
			where $\RMod_{\cW}$ is the category of right $\cW$-modules, the vertical arrow is the forgetful functor, and
			\[
			\free_{\cW} \colon \Rep(G) \rightarrow \RMod_{\cW}, \qquad V \mapsto V \otimes \cW
			\]
			is the functor of free right $\cW$-modules.
			\begin{proof}
				Follows from the PBW theorem and \cref{prop::subregular_walgebra_quotient}
			\end{proof}
		\end{prop}
		
		Observe that $\RMod_{\cW}$ is naturally a right module category over $\BiMod{\cW}{\cW}$, see \cref{def::module_categories}. Since $\res^{\psi}$ is a monoidal functor, it becomes a right module category over $\Rep(G)$ as well. Likewise, the category $\Rep(G)$ is tautologically a right module category over itself. Also, using Skryabin's theorem, we obtain a natural isomorphism
		\[
		X/\m^{\psi} \otimes_{\cW} (Y/\m^{\psi})^{\m^{\psi}} \cong X \otimes_{\U(\g)} \U(\g)/\m^{\psi} \otimes_{\cW} (Y/\m^{\psi})^{\m^{\psi}} \xrightarrow{\sim} X \otimes_{\U(\g)} Y/\m^{\psi}.
		\]
		In particular, for every $U,V \in \Rep(G)$, we have
		\begin{align*}
			\free_{\cW}(U) \otimes_{\cW} \res^{\psi}(\U(\g) \otimes V) \xrightarrow{\triv_U \otimes \id} \b \backslash (\U(\g) \otimes U) /\m^{\psi} \otimes_{\cW} (\U(\g) \otimes V/\m^{\psi})^{\m^{\psi}} \cong \\
			\cong \b \backslash \U(\g) \otimes U \otimes V/\m^{\psi} \xrightarrow{\triv_{U\otimes V}^{-1}} U \otimes V \otimes \cW.
		\end{align*}
		canonically. At the same time, since $\free_{\cW}(U)$ is a free $\cW$-module, we also have a canonical isomorphism
		\[
		\free_{\cW}(U) \otimes_{\cW} \res^{\psi}(\U(\g) \otimes V) = (U \otimes \cW) \otimes_{\cW} \res^{\psi}(\U(\g) \otimes V) \cong U \otimes V \otimes \cW 
		\]
		of right $\cW$-modules. Combining it with \cref{cor::ds_trivialization} and \cref{thm::ds_reduction_monoidal}, we get a ``matrix'' form of the monoidal structure on the Drinfeld-Sokolov reduction.
		\begin{thm}
			\label{thm::ds_reduction_tensor_structure_trivialization}
			The functor $\free_{\cW}\colon \Rep(G) \rightarrow \RMod_{\cW}$ is a functor of right $\Rep(G)$-module categories in the sense of \cref{def::module_categories_functor}. In particular, there is a collection of natural isomorphisms
			\[
			J_{UV}\colon U \otimes V \otimes \cW \rightarrow U \otimes V \otimes \cW,
			\]
			for all $U,V \in \Rep(G)$.
		\end{thm}
		
		In what follows, we will compute its semi-classical limit.

		\subsection{Whittaker vectors: vector representation}
		\label{subsect::whittaker_vectors_standard_representation}
		
		We explicitly compute the generating Whittaker vectors for $\U(\g) \otimes \C^N/\m^{\psi}$, where 
		\[
		\C^N = \Span(v_i | 1 \leq i \leq N), \qquad \ad_{E_{ij}}(v_k) = \delta_{jk} v_i
		\]
		(we use the notation $\ad_{E_{ij}}$ from \cref{subsect::hc_bimodules}).
		
		Recall the truncated generators \eqref{eq::bk_t_elements_truncated}. The next proposition gives a relation between ${}_{k}T$ for different values of $k$.
		\begin{prop}\cite[Lemma 10.4]{BrundanKleshchev}
			\label{prop:prop:Trelation}
			Suppose that $r>0$. Then
			\begin{equation}
				\label{eq::t_relation}
				{}_1T_{i,2;1}^{(r)} = {}_2 T_{i,2;1}^{(r)} + {}_2 T_{i,2;1}^{(r-1)} \widetilde{E}_{N-1,N-1} + [{}_2 T_{i,2;1}^{(r-1)} , \widetilde{E}_{N-2,N-1}]
			\end{equation}
			for $i=1,2$, where $[,]$ refers to the adjoint action from \cref{rmk::whittaker_adjoint}.
		\end{prop}
		We will need the following lemma.
		\begin{lm}
			\label{lm:ecommute}
			For $i=1,2$, we have
			\begin{equation}
				[E_{N,N-1}, {}_1T_{i,2;1}^{(r)}] = {}_2 T_{i,2;1}^{(r-1)}.
			\end{equation}
			\begin{proof}
				By \cref{prop:prop:Trelation}, it suffices to compute
				\begin{equation}
					\label{eq::prop_big_commutator}
					\left[E_{N,N-1}, \text{ } {}_{2}T_{i2;1}^{(r)}+{}_{2}T_{i2;1}^{(r-1)}\tilde{E}_{N-1,N-1}+[{}_{2}T_{i2;1}^{(r-1)},\tilde{E}_{N-2,N-1}]\right].
				\end{equation}
				Since ${}_{2}T_{i,2;x}^{(s)} \in \gl_{{}_2 N}$ and ${}_2 N < N-1$, we have $[E_{N,N-1},{}_{2}T_{i,2;x}^{(s)}]=0$ for any $s$. Thus, \eqref{eq::prop_big_commutator} becomes $$\left[E_{N,N-1}, \text{ } {}_{2}T_{i,2;1}^{(r-1)}\widetilde{E}_{N-1,N-1}+[{}_{2}T_{i,2;1}^{(r-1)},\widetilde{E}_{N-2,N-1}]\right]={}_{2}T_{i,2;1}^{(r-1)}+\left[E_{N,N-1}, [{}_{2}T_{i,2;1}^{(r-1)},\widetilde{E}_{N-2,N-1}]\right]$$ since $\widetilde{E}_{N-1,N-1} = E_{N-1,N-1} + \hbar \cdot c$ for some constant $c$ by \eqref{eq::modified_generators}. From Jacobi's identity, we have $$\left[E_{N,N-1},[{}_{2}T_{i,2;1}^{(r-1)},\widetilde{E}_{N-2,N-1}]\right]=-\left[{}_{2}T_{i,2;1}^{(r-1)},[E_{N,N-1},\tilde{E}_{N-2,N-1}]\right]-\left[\tilde{E}_{N-2,N-1},[E_{N,N-1},{}_{2}T_{i,2;1}^{(r-1)}]\right] = 0,$$ which proves the proposition.
			\end{proof}
		\end{lm}
		
		We go on to the main theorem.
		
		\begin{thm}
			\label{thm:whittaker_geq2}
			For $N-j \neq 1$, the following vectors in $\U(\g) \otimes \C^N/\m^{\psi}$ 
			\begin{equation}
				\label{eq::whittaker_vector_geq2}
				\tilde{v}_{N-j}^{\psi}=1 \otimes v_{N-j}+\sum_{i=0}^{j-1}(-1)^{j-i} {}_{i+1}T_{22;1}^{(j-i)} \otimes v_{N-i}
			\end{equation}
			are Whittaker.
			\begin{proof}
				We proceed with strong induction on the subregular pyramid $\pi$ from \eqref{eq::subregular_pyramid}. The base case is 
				\[
				\begin{picture}(30, 30)%
					\put(0,0){\line(1,0){30}}
					\put(0,15){\line(1,0){30}}
					\put(0,30){\line(1,0){30}}
					\put(0,0){\line(0,1){30}}
					\put(30,0){\line(0,1){30}}
					\put(15,23){\makebox(0,0){1}}
					\put(15,8){\makebox(0,0){2}}
					\put(-20,12){\makebox(0,0){$\pi =$}}
				\end{picture}
				\]
				and the corresponding nilpotent element with the nilpotent subalgebra are trivial. In particular, 
				\[1 \otimes v_2 \in \U(\gl_2) \otimes \C^2\]
				is automatically Whittaker.
				
				For the step, let assume that $N>2$ and we proved the statement for the truncated pyramid ${}_1 \pi$. It is clear that $1 \otimes v_N$ is Whittaker. Also, the vector
				\begin{equation}
					\label{eq::truncated_whittaker_vector}
					{}_1 \tilde{v}_{N-j}^{\psi} = 1 \otimes v_{N-j}  + \sum\limits_{i=1}^{j-1} (-1)^{j-i} {}_{i+1} T_{22;1}^{(j-i)} \otimes  v_{N-i} 
				\end{equation}
				is invariant under the truncated subalgebra ${}_1 \m$, recall \cref{def::subalgebras_truncation}. Indeed: while the coefficients of \eqref{eq::truncated_whittaker_vector} are different from the ones of \eqref{eq::whittaker_vector_geq2} for $\gl_{{}_1 N}$ by definition of the truncated generators \eqref{eq::bk_t_elements_truncated}, the non-standard embeddings $\U(\gl_{{}_k N}) \rightarrow \U(\gl_N)$ from \eqref{eq::nonstandard_embedding} are homomorphisms for all $1\leq k \leq N-1$, so the Whittaker property is preserved. Hence, let us rewrite equation \eqref{eq::whittaker_vector_geq2} in a recursive form:
				\begin{equation}
					\label{eq::thm_inductive_vector}
					\widetilde{v}_{N-j}^{\psi}={}_{1}\widetilde{v}_{N-j}^{\psi}+(-1)^{j}\cdot {}_{1}T_{22;1}^{(j)} \otimes v_{N}.
				\end{equation}
				To show that this vector is Whittaker, it suffices to check 
				\[
				(E_{N,N-1} - 1) \cdot \widetilde{v}_{N-j}^{\psi} = \hbar [E_{N,N-1},\widetilde{v}_{N-j}^{\psi}] = 0.
				\]
				(recall \cref{rmk::whittaker_adjoint}). Indeed,
				\begin{itemize}
					\item For all $x\in {}_1 \m$, we have 
					\[
					x^{\psi}\cdot {}_{1}\widetilde{v}_{N-j}^{\psi} = 0
					\]
					by induction hypothesis, and 
					\[
					x^{\psi}\cdot {}_{1}T_{22;1}^{(j)} \otimes v_{N} = 0
					\]
					by \cref{thm::subregular_walgebra_basis} and because any element of ${}_1\m$ commutes with $v_N$.
					\item For any $1 \leq k < N-1$, there exists $x\in {}_1 \m$, such that $E_{N,N-k} = [E_{N,N-1}^{\psi},x^{\psi}]$. Therefore, assuming we proved invariance under $E_{N,N-1}^{\psi}$, we have
					\[
					E_{N,N-k}^{\psi}\cdot \tilde{v}_{N-j}^{\psi} = \hbar^{-1}(E_{N,N-1}^{\psi} x^{\psi} - x^{\psi} E_{N,N-1}^{\psi}) \tilde{v}_{N-j}^{\psi} = 0.
					\]
				\end{itemize}

				By construction,
				\[
				{}_{1}\widetilde{v}_{N-j}^{\psi}={}_{2}\widetilde{v}_{N-j}^{\psi}+(-1)^{j-1}\cdot {}_{2}T_{22;1}^{(j-1)} \otimes v_{N-1}.
				\] 
				Since ${}_2 \tilde{v}_{N-j}^{\psi} \in \U(\gl_{N-2}) \otimes \C^{N-2}$, we have $[E_{N,N-1},{}_2 \tilde{v}_{N-j}^{\psi}] = 0$. Likewise,
				\[
				[E_{N,N-1},{}_2 T_{22;1}^{(j-1)} \otimes v_{N-1}] = {}_2 T_{22;1}^{(j-1)} \otimes v_N,
				\]
				and therefore,
				\[
				[E_{N,N-1},{}_{1}\widetilde{v}_{N-j}^{\psi}] = (-1)^{j-1} \cdot {}_2 T_{22;1}^{(j-1)} \otimes  v_{N} .
				\]
				By \cref{lm:ecommute}, we get
				\[
				[E_{N,N-1},(-1)^j  \cdot {}_1 T_{22;1}^{(j)} \otimes v_N] = (-1)^j {}_2 T_{22;1}^{(j-1)} \otimes v_N .
				\]
				Summing up these equalities and recalling \eqref{eq::thm_inductive_vector}, we conclude that 
				\[
				[E_{N,N-1}, \widetilde{v}_{N-j}^{\psi}] = 0,
				\]
				and the induction is complete.
			\end{proof}
		\end{thm}
		
		\begin{thm}
			\label{thm:whittaker_geq3}
			The remaining vector 
			\begin{equation}
				\label{eq::whittaker_vector_1}
				\tilde{v}_1^{\psi}=1 \otimes v_{1} +\sum_{i=0}^{N-3}(-1)^{N-i-2}\cdot {}_{i+1}T_{12;1}^{(N-i-2)} \otimes v_{N-i} 
			\end{equation}
			is also Whittaker in $\U(\g) \otimes \C^N/\m^{\psi}$.
			\begin{proof}
				Similar to \cref{thm:whittaker_geq2}.
			\end{proof}
		\end{thm}
		
		Observe that these vectors do not quite satisfy the assumption of \eqref{eq::whittaker_vector_canonical_form}. However, they are not far away from the canonical form. Recall the definition of the algebra $\l$ from \eqref{eq::l_subalgebra}. Observe that $\l \subset \cW$.
		\begin{defn}
			\label{def::l_constant}
			Consider the natural PBW basis of $\U(\g)$ induced from the basis $\{E_{ij}\}$ of $\g$. An element $x\in \U(\g)$ is called \defterm{$\l$-constant}, if $x\in \U(\l)$.
		\end{defn}
		
		Thanks to upper-triangular form of \eqref{eq::whittaker_vector_geq2} and \eqref{eq::whittaker_vector_1}, we can apply some strictly upper-triangular (hence invertible) matrix with coefficients in $\U(\l)$ to the constructed generators to bring it to the necessary form. Namely, denote by $c_{N-j}^{N-i}$ the $\l$-constant term of $(-1)^{j-i} \cdot {}_{i+1} T_{22;1}^{(j-i)}$. Then we can perform the following inductive operation:
		\begin{equation}
			\widetilde{v}_{N-j}^{\psi} \mapsto \widetilde{v}_{N-j}^{\psi} - \sum\limits_{i=0}^{j-1} v_{N-i}^{\psi} c_{N-j}^{N-i} .
		\end{equation}
		Removing step-by-step all the $\l$-constant terms, we eventually get the canonical generators, that we denote by 
		\begin{equation}
			\label{eq::whittaker_canonical_gens}
			v_i^{\psi} = 1\otimes v_i + \sum\limits_{j>i} x_i^j \otimes v_j,\ x_i^j \in \b \cdot \U(\g),
		\end{equation}
		for $1\leq i \leq N$.
		
		Moreover, this form is actually more refined: we have
		\begin{equation}	
			v_i^{\psi} = 1 \otimes v_i + \sum\limits_{j>i} x_i^j \otimes v_j,\ x_i^j \in \b_{j-1}\cdot \U(\g),
		\end{equation}
		where $\b_{j-1} \subset \gl_{j-1}$ is the truncated Borel subalgebra as in \cref{def::borel_truncation}.
		
		\subsection{Tensor structure}
		\label{sect::tensor}
		The goal of this section is to compute the semi-classical limit of the monoidal isomorphism from \cref{thm::ds_reduction_tensor_structure_trivialization} for the tensor product of $\C^N$ with itself. From now on, we will treat $\hbar$ as a \emph{variable}, in particular, we consider the asymptotic universal enveloping algebra $\U(\gl_N)$ over $\C[\hbar]$. We will need some definitions regarding ``asymptotic'' behavior of elements of $\U(\g)$.
		
		\begin{defn}
			\label{def::asymptotically_linear}
			Consider the natural PBW basis of $\U(\g)$ induced from the basis $\{E_{ij}\}$ of $\g$. We call an element $x\in \U(\g)$ \defterm{constant} if it has degree zero with respect to this basis. It is called \defterm{asymptotically linear} if the PBW degree of $x$ is one and it is constant in $\hbar$. We call $x$ \defterm{asymptotically $\l$-linear} if it constant in $\hbar$ and has the form $x \in y \cdot \U(\l)^{>0}$, where $y \in \b$.
		\end{defn}
		
		\cref{thm::ds_reduction_tensor_structure_trivialization} in this particular case can be reformulated as follows. Let $\{v_i \otimes v_j\}$ be a natural basis of $\C^N \otimes \C^N$. We have two natural choices of generating vectors in $(\U(\g) \otimes \C^N \otimes \C^N /\m^{\psi})^{\m^{\psi}}$: one is provided by \cref{cor::ds_trivialization}, and we denote it by
		\begin{equation}
			\label{eq::canonical_tensor_product_generators}
			(v_i \otimes v_j)^{\psi} = v_i \otimes v_j + \sum\limits_{k,l} x_{ij}^{kl} \otimes v_k \otimes v_l,\ x_{ij}^{kl} \in \b\cdot \U(\p).
		\end{equation}
		Another is given by the monoidal structure \eqref{eq::explicit_tensor_structure} on the Drinfeld-Sokolov reduction: under canonical trivialization \eqref{eq::whittaker_canonical_gens}, we set
		\begin{align*}
			v_i^{\psi} \otimes v_j^{\psi} := v_i^{\psi} \otimes_{\U(\g)} v_j^{\psi} \in &(\U(\g) \otimes \C^N /\m^{\psi})^{\m^{\psi}} \otimes_{\cW} ( \U(\g) \otimes \C^N /\m^{\psi})^{\m^{\psi}} \cong  \\
			\cong &(\U(\g) \otimes \C^N \otimes \C^N/\m^{\psi})^{\m^{\psi}}.
		\end{align*}
		
		\begin{prop}
			\label{prop::monoidal_iso_vector_reps}
			The monoidal isomorphisms $J_{\C^N,\C^N}$ from \cref{thm::ds_reduction_tensor_structure_trivialization} are of the form
			\[
			J_{\C^N, \C^N} \in \id_{\C^N \otimes \C^N \otimes \cW} + \hbar \U(\b)^{>0} \otimes \U(\m)^{>0} \otimes \U(\l),
			\]
			where $\U(\l) \subset \cW$.
			\begin{proof}
				Recall that under identification $\U(\g)/\m^{\psi} \cong \U(\p)$, the generating vectors have the form \eqref{eq::whittaker_canonical_gens}, so,
				\begin{equation}
					\label{eq::prop_tensor_product_1}
					v_i^{\psi} \otimes v_j^{\psi} = \left( 1 \otimes v_i + \sum\limits_{k>i} x_i^k \otimes  v_k \right) \otimes \left(1\otimes v_j + \sum\limits_{l>j} x_j^l \otimes v_l \right).
				\end{equation}
				It follows from construction that for all $j,l$, 
				\[
				x_j^l  = (x_{(1)})_j^l \cdot (x_{(2)})_j^l
				\]
				for some $(x_{(1)})_j^l \in \U(\b)$ and $(x_{(2)})_j^l\in \U(\l)$ (here, we use Sweedler's sum notation). Moreover, observe that 
				\[
				(x_{(2)})_j^l \otimes v_l = (1\otimes v_l) (x_{(2)})_j^l
				\]
				(recall the right action from \cref{subsect::hc_bimodules}). In particular, we get
				\[
				(1\otimes v_k) \cdot x_j^l = \sum\limits_{a\leq k} ((x_{(1)})_{jk}^{la} \otimes v_a ) (x_{(2)})_j^l
				\]
				for some $(x_{(1)})_{jk}^{la} \in \U(\b)$. Therefore, 
				\begin{align}
					\label{eq::prop_tensor_product_2}
					\begin{split}
						v_i^{\psi} \otimes v_j^{\psi} &= 1 \otimes v_i \otimes v_j + \sum\limits_{k > i} x_i^k  \otimes v_k \otimes v_j + \sum\limits_{\substack{l>j \\ k > i}} \sum\limits_{b \leq k} (x_i^k (x_{(1)})_{jk}^{lb} \otimes v_b \otimes v_l) (x_{(2)})_{j}^l + \\
						+& \sum\limits_{\substack{a \leq i \\ l > j}} ((x_{(1)})_{ji}^{la} \otimes v_a \otimes v_l) (x_{(2)})_j^l.
					\end{split}
				\end{align}
				Observe that the second line already has the form \eqref{eq::whittaker_vector_canonical_form}, and the third line almost satisfies this condition as well except for the case when $(x_{(1)})_{ji}^{la}$ is constant; in this case, denote $(x_{(1)})_{ji}^{la} \cdot (x_{(2)})_j^l =: c_{ij}^{al}\in \U(\l)$. We see that the map
				\[
				v_i \otimes v_j \otimes 1 \mapsto v_i \otimes v_j \otimes 1 + \sum_{\substack{a \leq i \\ l > j}} v_a \otimes v_l \otimes c_{ij}^{al}
				\]
				is strictly upper-triangular. In particular, the canonical generators $(v_i \otimes v_j)^{\psi}$ of \eqref{eq::canonical_tensor_product_generators} can be constructed inductively by taking 
				\[
				(v_i \otimes v_j)^{\psi} \otimes 1 := v_i^{\psi} \otimes v_j^{\psi} \otimes 1 - \sum  (v_a \otimes v_l)^{\psi} \otimes c_{ij}^{al}.
				\]
				Then the tensor isomorphism is given by a $\U(\l)$-valued matrix 
				\begin{equation}
					\label{eq::tensor_matrix}
					J_{\C^N,\C^N} = \id_{\C^N \otimes \C^N \otimes \cW} + (c_{ij}^{al}).
				\end{equation}
				
				Moreover, since every commutation produces a power of $\hbar$ by \eqref{eq::hc_right_action}, the second part of the theorem follows.
			\end{proof}
		\end{prop}

		We will compute the first $\hbar$-power $\mathbf{j}$ of the matrix \eqref{eq::tensor_matrix}. In other words, we need to compute the first $\hbar$-powers of the coefficients $(x_j^l)_i^a$ of \eqref{eq::prop_tensor_product_2}.

		\begin{prop}
			\label{prop::asymptotically_linear_contribution}
			The only non-trivial contribution to the first $\hbar$-power from $(x_j^l)_i^a$ comes from asymptotically linear and $\l$-linear terms of $x_j^l$ in \eqref{eq::prop_tensor_product_1}.
			\begin{proof}
				Indeed: in a PBW basis, the element $x_j^l$ is the sum of products of the form $\hbar^b y_1\cdots y_a\cdot x$ with $a\geq 1$ for some $\{y_m \} \subset \g$, some $\hbar$-power $b$, and $x\in \U(\l)$. If it is not asymptotically ($\l$)-linear, then there are two cases:
				\begin{enumerate}
					\item It is not linear, i.e. $a\geq 2$; for simplicity, we demonstrate it for $a=2$, but the general argument is the same:
					$$(1\otimes v_i) \cdot \hbar^b y_1 y_2 x= \hbar^b (y_1 \otimes v_i) y_2x - (\hbar^{b+1} \otimes [y_1,v_i])y_2 x$$
					$$=\hbar^b (y_2 y_1 \otimes v_i) x- \hbar^{b+1} ([y_2,y_1] \otimes v_i )x- \hbar^{b+1} (y_1 \otimes [y_2,v_i] )x-\hbar^{b+1}(y_2 \otimes [y_1,v_i])x + (\hbar^{b+2} \otimes [y_2,[y_1,v_i]])x.$$
					So, we see that there is no contribution to the first $\hbar$-power of constant terms for $a\geq 2$.
					\item It is linear, but divisible by $\hbar$, i.e. $b\geq 1$. Then 
					\[
					(1 \otimes v_i) \hbar^b y_1x = \hbar^b (y_1 \otimes v_i)x - (\hbar^{b+1} \otimes [y_1,v_i])x,
					\]
					and there is also no contribution to the first $\hbar$-power of constant terms.
				\end{enumerate}
				Thus, if $x_{j}^{l}$ is not asymptotically ($\l$)-linear, then it must be asymptotically linear for there to be non-trivial contribution to the first $\hbar$ power. At the same time, the calculation above shows that for $b=0$, $\l$-linear terms may contribute to the first power, and the proposition follows.   
			\end{proof}
		\end{prop}
		
		Unfortunately, the only explicit form of Whittaker vectors in $\U(\g) \otimes \C^N /\m^{\psi}$ available so far is \eqref{eq::whittaker_vector_geq2} or \eqref{eq::whittaker_vector_1} which is not canonical; however, thanks to the next lemmas, it does not affect calculations too much.
		\begin{lm}
			\label{lm::constant_terms_t}
			For any parameters, the $\l$-constant part of ${}_a T_{ij;x}^{(r)}$ is divisible by $\hbar$.
			\begin{proof}
				From \cref{def:t_els}, we note that a constant term exists in $T_{ij;x}^{(r)}$ if and only if $i_k=j_k$ for all $1 \le k \le s$. But from condition \eqref{item::degree}, we note that $\text{col}(j_k)-\text{col}(i_k)+1=1$ for all $k$. Thus, $s=r$ and constant terms only come from summands of the form
				\[
				\widetilde{E}_{i_1 i_1} \cdot \ldots \cdot \widetilde{E}_{i_r i_r}.
				\] By \eqref{eq::modified_generators}, it is clear that the constant term is proportional to $\hbar^r$.
				
				As for general $\l$-constant terms, it follows from the formula \eqref{eq::bk_t_elements} that $T_{22;1}^{(r)}$ is the sum of elements of the form
				\[
				\widetilde{E}_{2,1} \widetilde{E}_{1,1}^k \cdot x,\ x\in \U(\b)
				\]
				for some $k$. Note $x$ must commute with $E_{11}$ from condition \eqref{item::second}. Thus, commuting $x$ to the left produces some elements of $\U(\l)$, but they are divisible by $\hbar$ because of commutation. The same is true for $T_{12;1}^{(r)}$, where the terms have the form $\widetilde{E}_{1,1}^k\cdot x$ for some $k$ and $x\in \U(\b)$.
			\end{proof}
		\end{lm}
		
		Recall the coefficients $x_{N-j}^{N-i}$ from \eqref{eq::whittaker_canonical_gens}.
		
		\begin{lm}
			\label{lm::asymptotically_linear_terms_same}
			The asymptotically $\l$-linear terms in $(-1)^{j-i} \cdot {}_{i+1} T_{22;1}^{(j-i)}$ are the same as in $x_{N-j}^{N-i}$ for $N-j \neq 1$. Likewise, the asymptotically $\l$-linear terms in ${}_{i+1} T_{12;1}^{(N-i-1)}$ are the same as $x_1^{N-l}$.
			\begin{proof}
				Recall \eqref{eq::whittaker_canonical_gens} that the canonical generators can be constructed from $\{\widetilde{v}_i^{\psi}\}$ by inductively removing $\l$-constant terms. But, according to \cref{lm::constant_terms_t}, they are all divisible by $\hbar$, and the statement follows.
			\end{proof}
		\end{lm}
		
		Therefore, by \cref{prop::asymptotically_linear_contribution}, it is enough to consider only the asymptotically linear terms of $T$-generators.
		\begin{prop} The explicit forms for the asymptotically linear and $\l$-linear terms are given below.
			\label{prop::asymptotically_linear_part}
			\begin{itemize}
				\item The asymptotically linear part of $T_{ij;x}^{(r)}$ is 
				\[
				\sum\limits_{\substack{\row(i_1) = i, \\ \row(j_1) = j, \\ \col(j_1)- \col(i_1) + 1 = r}} (-1)^{r-1} E_{i_1,j_1}.
				\]
				\item The asymptotically $\l$-linear terms of ${}_{i+1}T_{22;1}^{(j-i)}$ are 
				\[
				\sum\limits_{r=2}^{j-i} (-1)^{r}  E_{1,r} E_{2,1} E_{1,1}^{j-i-r}.
				\]
				\item The asymptotically $\l$-linear terms of ${}_{i+1}T_{12;1}^{(N-i-1)}$ are
				\[
				\sum\limits_{r=2}^{N-i-2} (-1)^{r} E_{1,r} E_{1,1}^{N-i-r-1}. 
				\]
			\end{itemize}
			
			\begin{proof}
				Recall the formula \eqref{eq::bk_t_elements} of $T$-generators. Observe that in order to have a (not necessarily asymptotically) linear term in a summand with $s$ terms, we need at least $s-1$ of those $\widetilde{E}_{i_l,j_l}$ to carry a constant term. Hence, we must have $i_l=j_l$ for at least $s-1$ values of $l$ where $1 \le l \le s$. But then $\hbar$ divides each of these constant terms by \eqref{eq::modified_generators}, so we require $s-1=0$. Therefore, only the linear part
				\[
				\sum\limits_{\substack{\row(i_1) = i, \\ \row(j_1) = j, \\ \col(j_1)- \col(i_1) + 1 = r}} \tilde{E}_{i_1,j_1}
				\]
				can contribute to the first power of $\hbar$, which is precisely the formula from the statement.
				
				Now let us study the asymptotically $\l$-linear terms of ${}_{i+1} T_{22;1}^{(j-i)}$. Consider a summand of \eqref{eq::bk_t_elements}. By condition \eqref{item::first}, $\row(i_1) = 2$. Assume that $\col(i_1)>1$. Then by condition \eqref{item::bk_t_parabolic}, $\col(j_1)>1$, and so $\row(j_1) = 2$. In particular, $\sigma_{\row(j_1)} = +$ meaning that $\col(i_2)>1$ by condition \eqref{item::fifth}. Continuing, we obtain that this summand cannot contain $E_{1,1}$ or $E_{2,1}$.
				
				Consider $\col(i_1)=1$, i.e. $\widetilde{E}_{i_1,j_1} = E_{2,1}$. By condition \eqref{item::second}, $\row(j_1)=1 = \row(i_2)$. If $\row(j_2) > 1$ and we can apply previous arguments to conclude that the corresponding summand is $E_{2,1} \widetilde{E}_{1,r}$ (recall that we are interested only in $\l$-linear terms). Otherwise, by condition \eqref{item::sixth}, we see that $i_3 = 1$, and we can repeat the argument. Summing all cases, we obtain a summand of the form (observe that we drop all the $\hbar$-factors)
				\[
				(-1)^r E_{2,1} E_{1,1}^k E_{1,r}.
				\]
				Now we commute the $\l$-part to the right. Observe that it produces powers of $\hbar$, so,  
				\[
				(-1)^r E_{2,1} E_{1,1}^k E_{1,r} = (-1)^r E_{1,r} E_{2,1} E_{1,1}^k + O(\hbar).
				\]
				The relation between $r$ and $k$ follows from the degree condition \eqref{item::degree}.
				
				The analysis for ${}_{i+1} T_{12;1}^{(N-i-1)}$ is similar and will be omitted.
			\end{proof}
		\end{prop}
		
		Combining all preliminary results, we can compute the tensor isomorphism \cref{thm::ds_reduction_tensor_structure_trivialization} for the vector representation.
		\begin{prop}
			\label{prop::semiclassical_limit_vector}
			The monoidal isomorphism $J_{\C^N,\C^N}$ from \cref{prop::monoidal_iso_vector_reps} has the form
			\[
			J_{\C^N,\C^N} = \id_{\C^N \otimes \C^N \otimes \cW} + \hbar \mathbf{j}_{\C^N,\C^N} + O(\hbar^2),
			\]
			where $\mathbf{j}_{\C^N,\C^N}$ is
			\begin{align*}
				\mathbf{j}_{\C^N,\C^N} = \mathbf{j}_c + \sum\limits_{j=2}^{N-2} \sum\limits_{i=j+2}^{N} \sum\limits_{r=2}^{i-j} (-1)^{i-j-r} x_{21} x_{11}^{i-j-r} E_{1,r} \otimes E_{i,j} + \sum\limits_{i=4}^{N} \sum\limits_{r=2}^{i-2} (-1)^{i-r} x_{11}^{i-r-1} E_{1,r} \otimes E_{i,1} 
			\end{align*}
			with $x_{21},x_{11}$ the coordinate functions on $\l^*$ corresponding to $E_{21},E_{11} \in \l$. Here,
			\[
			\mathbf{j}_c = \sum\limits_{j=2}^{N-1} \sum\limits_{i=j+1}^{N} \left( \sum\limits_{l=2}^{j} E_{l,l+i-j-1} \right) \otimes E_{i,j} + \sum\limits_{i=3}^{N} E_{1,i-1} \otimes E_{i,1}
			\]
			is a map $\b^* \rightarrow \m$ from \cref{prop::wonderbolic_r_matrix}.
			\begin{proof}
				Denote by $L_{N-j}^{N-k}$ the asymptotically linear part of 
				\begin{align*}
					(-1)^{j-k} {}_{k+1} T_{22;1}^{(j-k)},\ &N-j\neq 1, \\
					(-1)^{N-k-2} \cdot {}_{k+1} T_{12;1}^{(N-k-2)},\ &N-j=1.
				\end{align*}
				It follows from \cref{prop::asymptotically_linear_part} that 
				\begin{align}
					\begin{split}
						L_{N-j}^{N-k} &= -\sum\limits_{l=2}^{N-j} E_{l,l+j-k-1}, \\
						L_{1}^{N-k} &= -E_{1,N-k-1}.
					\end{split}
				\end{align}
				Combing \cref{eq::prop_tensor_product_2}, \cref{lm::constant_terms_t}, and \cref{lm::asymptotically_linear_terms_same}, the constant part of the first $\hbar$-power of $J_{\C^N,\C^N}$ is given by the action of 
				\[
				\sum\limits_{j=1}^{N-2} \sum\limits_{k=0}^{j-1} \left( \sum\limits_{l=2}^{N-j} E_{l,l+j-k-1} \right) \otimes E_{N-k,N-j} + \sum\limits_{k=0}^{N-3} E_{1,N-k-1} \otimes E_{N-k,1}.
				\]
				Changing coefficients:
				\[
				\sum\limits_{j=2}^{N-1} \sum\limits_{i=j+1}^{N} \left( \sum\limits_{l=2}^{j} E_{l,l+i-j-1} \right) \otimes E_{i,j} + \sum\limits_{i=3}^{N} E_{1,i-1} \otimes E_{i,1}.
				\]
				Now let us consider the ``dynamical'' part. It follows from \cref{prop::asymptotically_linear_part} that the asymptotically $\l$-linear part of the coefficient $(-1)^{j-i} \cdot {}_{i+1} T_{22;1}^{(j-i)}$ of $v_{N-i}$ in $\widetilde{v}_{N-j}^{\psi}$ is 
				\[
				\sum\limits_{r=2}^{j-i} (-1)^{j-i+r} E_{1,r} E_{2,1} E_{1,1}^{j-i-r} .
				\]
				Recall from the construction of the tensor structure from \cref{prop::monoidal_iso_vector_reps} that we need to compute the first $\hbar$-power of the right action
				\[
				\sum\limits_{r=2}^{j-i} (-1)^{j-i+r-1} (1 \otimes v_k) \cdot E_{1,r} E_{21} E_{11}^{j-i-r}
				\]
				modulo $\b$ for every $k$. It follows that it is equal to
				\[
				\sum\limits_{r=2}^{j-i} \left((-1)^{j-i+r}  \otimes \ad_{E_{1,r}}(v_k)\right) E_{2,1} E_{1,1}^{j-i-r},
				\]
				which for all admissible $i,j$, gives the contribution
				\[
				\sum\limits_{j=2}^{N-2} \sum\limits_{i=j+2}^{N} \sum\limits_{r=2}^{i-j} (-1)^{i-j-r} x_{21} x_{11}^{i-j-r}E_{1,r} \otimes E_{i,j}. 
				\]
				Likewise, by considering the coefficients of $v_1^{\psi}$, we get the contribution
				\[
				\sum\limits_{i=4}^{N} \sum\limits_{r=2}^{i-1} (-1)^{i-r} x_{11}^{i-r} E_{1,r} \otimes E_{i,1},
				\]
				and the theorem follows.
			\end{proof}
		\end{prop}
		
		In fact, the constant part $\mathbf{j}_c$ is related to the form $\omega$ from \cref{prop::wonderbolic_r_matrix}.
		
		\begin{prop}
			The inverse $\omega^{-1}$ is equal to $\mathbf{j}_c-\mathbf{j}_c^{21}$.
		\end{prop}
		
		\begin{proof}
			One can easily see that the conditions \eqref{eq::subregular_r_matrix_conditions} are satisfied.
		\end{proof}
		
		Finally, by the Schur-Weyl duality, any representation of $\GL_N$ can be canonically obtained as a subrepresentation of $(\C^N)^{\otimes k} \otimes \det^l$ for some $k,l$, where $\det$ is the one-dimensional determinant representation. By naturality of construction, we obtain the main result of the paper.
		\begin{thm}
			\label{thm::semiclassical_limit}
			The semi-classical limit of the monoidal isomorphisms $J_{UV}$ from \cref{thm::ds_reduction_tensor_structure_trivialization} is given by the action of the universal element
			\[
			\mathbf{j} = \mathbf{j}_c + \sum\limits_{j=2}^{N-2} \sum\limits_{i=j+2}^{N} \sum\limits_{r=2}^{i-j} (-1)^{i-j-r} x_{21} x_{11}^{i-j-r} E_{1,r} \otimes E_{i,j} + \sum\limits_{i=4}^{N} \sum\limits_{r=2}^{i-2} (-1)^{i-r} x_{11}^{i-r-1} E_{1,r} \otimes E_{i,1},
			\]
			where 
			\[
			\mathbf{j}_c = \sum\limits_{j=2}^{N-1} \sum\limits_{i=j+1}^{N} \left( \sum\limits_{l=2}^{j} E_{l,l+i-j-1} \right) \otimes E_{i,j} + \sum\limits_{i=3}^{N} E_{1,i-1} \otimes E_{i,1}
			\]
			defines an inverse of $\omega$ from \cref{prop::wonderbolic_r_matrix}.
		\end{thm}
		
		\printbibliography
		
	\end{document}